\documentclass[12pt]{amsart}

\usepackage[english]{babel}
\usepackage{bm}
\usepackage{fullpage}
\usepackage{amssymb}
\usepackage{amsmath}
\usepackage{latexsym}
\usepackage{graphicx}
\usepackage{algorithm}

\newtheorem{lemma}{Lemma}

\newtheorem{proposition}{Proposition}
\newtheorem{theorem}{Theorem}
\newtheorem*{theorem-cite}{Theorem}

 \theoremstyle{remark}

\newtheorem*{example}{Example}

\def\KG{\operatorname{KG}}
\def\cd{\operatorname{cd}}
\def\alt{\operatorname{alt}}
\def\hyp{\mathcal{H}}
\def\id{\operatorname{id}}

\def\sd{\operatorname{sd}}

\author{Fr\'ed\'eric Meunier}
\address{F. Meunier, Universit\'e Paris Est, CERMICS (ENPC) \\F-77455 Marne-la-Vall\'ee}
\email{frederic.meunier@enpc.fr}

\date{\today}

\title{Colorful hypergraphs in Kneser hypergraphs}

\begin{document}

\begin{abstract}
Using a $Z_q$-generalization of a theorem of Ky Fan, we extend to Kneser hypergraphs a theorem of Simonyi and Tardos that ensures the existence of multicolored complete bipartite graphs in any proper coloring of a Kneser graph. It allows to derive a lower bound for the local chromatic number of Kneser hypergraphs (using a natural definition of what can be the local chromatic number of a hypergraph).
\end{abstract}

\keywords{colorful complete $p$-partite hypergraph; combinatorial topology; Kneser hypergraphs; local chromatic number}

\maketitle

\section{Introduction}

\subsection{Motivations and results}

A {\em hypergraph} is a pair $\hyp=(V(\hyp),E(\hyp))$, where $V(\hyp)$ is a finite set and $E(\hyp)$ a family of subsets of $V(\hyp)$. The set $V(\hyp)$ is called the {\em vertex set} and the set $E(\hyp)$ is called the {\em edge set}. A {\em graph} is a hypergraph each edge of which is of cardinality two. A {\em $q$-uniform hypergraph} is a hypergraph each edge of which is of cardinality $q$. The notions of graphs and $2$-uniform hypergraphs therefore coincide. If a hypergraph has its vertex set partitioned into subsets $V_1,\ldots,V_q$ so that each edge intersects each $V_i$ at exactly one vertex, then it is called a {\em $q$-uniform $q$-partite hypergraph}. The sets $V_1,\ldots,V_q$ are called the {\em parts} of the hypergraph. When $q=2$, such a hypergraph is a graph and said to be {\em bipartite}. A $q$-uniform $q$-partite hypergraph is said to be {\em complete} if all possible edges exist.

A {\em coloring} of a hypergraph is a map $c:V(\hyp)\rightarrow[t]$ for some positive integer $t$. A coloring is said to be {\em proper} if there in no {\em monochromatic} edge, i.e. no edge $e$ with $|c(e)|=1$. The chromatic number of such a hypergraph, denoted $\chi(\hyp)$, is the minimal value of $t$ for which a proper coloring exists. Given $X\subseteq V(\hyp)$, the hypergraph with vertex set $X$ and with edge set $\{e\in E(\hyp):e\subseteq X\}$ is the {\em subhypergraph of $\hyp$ induced by $X$} and is denoted $\mathcal{H}[X]$. \\

Given a hypergraph $\hyp=(V(\hyp),E(\hyp))$, we define the {\em Kneser graph} $\KG^2(\hyp)$ by 
$$\begin{array}{rcl}
V(\KG^2(\hyp)) & = & E(\hyp) \\
E(\KG^2(\hyp)) & = & \{\{e,f\}:e,f\in E(\hyp),\,e\cap f=\emptyset\}.
\end{array}$$
The ``usual'' Kneser graphs, which have been extensively studied -- see \cite{St76,VaVe05} among many references, some of them being given elsewhere in the present paper -- are the special cases $\hyp=([n],{{[n]}\choose k})$ for some positive integers $n$ and $k$ with $n\geq 2k$. We denote them $\KG^2(n,k)$. The main result for ``usual'' Kneser graphs is Lov\'asz's theorem~\cite{Lo79}.
\begin{theorem-cite}[Lov\'asz theorem]
Given $n$ and $k$ two positive integers with $n\geq 2k$, we have $\chi(\KG^2(n,k))=n-2k+2$.
\end{theorem-cite}

The $2$-colorability defect $\cd^2(\hyp)$ of a hypergraph $\hyp$ has been introduced by Dol'nikov~\cite{Do88} in 1988 for a generalization of Lov\'asz's theorem. It is defined as the minimum number of vertices that must be removed from $\hyp$ so that the hypergraph induced by the remaining vertices is of chromatic number at most $2$:
$$\cd^2(\hyp)=\min\{|Y|:\,Y\subseteq V(\hyp), \chi(\hyp[V(\hyp)\setminus Y])\leq 2\}.$$
\begin{theorem-cite}[Dol'nikov theorem]
Let $\hyp$ be a hypergraph and assume that $\emptyset$ is not an edge of $\hyp$. Then $\chi(\KG^2(\hyp))\geq\cd^2(\hyp)$.
\end{theorem-cite}
It is a generalization of Lov\'asz theorem since $\cd^2([n],{{[n]}\choose k})=n-2k+2$ and since the inequality $\chi(\KG^2(n,k))\leq n-2k+2$ is the easy one. 

The following theorem proposed by Simonyi and Tardos in 2007~\cite{SiTa07} generalizes Dol'nikov's theorem. The special case for ``usual'' Kneser graphs is due to Ky Fan~\cite{Fa82}.
\begin{theorem-cite}[Simonyi-Tardos theorem]
Let $\hyp$ be a hypergraph and assume that $\emptyset$ is not an edge of $\hyp$. Let $r=\cd^2(\hyp)$. Then any proper coloring of $\KG^2(\hyp)$ with colors $1,\ldots,t$ ($t$ arbitrary) must contain a completely multicolored complete bipartite graph $K_{\lceil r/2\rceil,\lfloor r/2\rfloor}$ such that the $r$ different colors occur alternating on the two parts of the bipartite graph with respect to their natural order.
\end{theorem-cite}

In 1976, Erd\H{o}s~\cite{Er76} initiated the study of {\em Kneser hypergraphs} $\KG^q(\mathcal{H})$ defined for a hypergraph $\hyp=(V(\hyp),E(\hyp))$ and an integer $q\geq 2$ by
$$\begin{array}{rcl}
V(\KG^q(\hyp)) & = & E(\hyp) \\
E(\KG^q(\hyp)) & = & \{\{e_1,\ldots,e_q\}:e_1,\ldots,e_q\in E(\hyp),\,e_i\cap e_j=\emptyset\mbox{ for all $i,j$ with $i\neq j$}\}.
\end{array}$$ 
A Kneser hypergraph is thus the generalization of Kneser graphs obtained when the $2$-uniformity is replaced by the $q$-uniformity for an integer $q\geq 2$. There are also ``usual'' Kneser hypergraphs, which are obtained with the same hypergraph $\hyp$ as for ``usual'' Kneser graphs, i.e. $\hyp=([n],{{[n]}\choose k})$. They are denoted $\KG^q(n,k)$. The main result for them is the following generalization of Lov\'asz's theorem conjectured by Erd\H{o}s and proved by Alon, Frankl, and Lov\'asz~\cite{AlFrLo86}.

\begin{theorem-cite}[Alon-Frankl-Lov\'asz theorem]
Given $n$, $k$, and $q$ three positive integers with $n\geq qk$, we have $\chi(\KG^q(n,k))=\left\lceil\frac{n-q(k-1)}{q-1}\right\rceil$.
\end{theorem-cite}

There exists also a $q$-colorability defect $\cd^q(\hyp)$, introduced by K\v{r}\'i\v{z}, defined as the minimum number of vertices that must be removed from $\hyp$ so that the hypergraph induced by the remaining vertices is of chromatic number at most $q$:
$$\cd^q(\hyp)=\min\{|Y|:\,Y\subseteq V(\hyp), \chi(\hyp[V(\hyp)\setminus Y])\leq q\}.$$
The following theorem, due to K\v{r}\'i\v{z}~\cite{Kr92,Kr00}, generalizes Dol'nikov's theorem. It also generalizes the Alon-Frankl-Lov\'asz theorem since $\cd^q([n],{{[n]}\choose k})=n-q(k-1)$ and since again the inequality $\chi(\KG^q(n,k))\leq \left\lceil\frac{n-q(k-1)}{q-1}\right\rceil$ is the easy one. 
\begin{theorem-cite}[K\v{r}\'i\v{z} theorem]
Let $\hyp$ be a hypergraph and assume that $\emptyset$ is not an edge of $\hyp$. Then $$\chi(\KG^q(\hyp))\geq\left\lceil\frac{\cd^q(\hyp)}{q-1}\right\rceil$$ for any integer $q\geq 2$.
\end{theorem-cite}

Our main result is the following extension of Simonyi-Tardos's theorem to Kneser hypergraphs.

\begin{theorem}\label{thm:main}
Let $\hyp$ be a hypergraph and assume that $\emptyset$ is not an edge of $\hyp$. Let $p$ be a prime number. Then any proper coloring $c$ of $\KG^p(\hyp)$ with colors $1,\ldots,t$ ($t$ arbitrary) must contain a complete $p$-uniform $p$-partite hypergraph with parts $U_1,\ldots,U_p$ satisfying the following properties. 
\begin{itemize}
\item It has $\cd^p(\hyp)$ vertices. 
\item The values of $|U_j|$ for $j=1,\ldots,p$ differ by at most one.
\item For any $j$, the vertices of $U_j$ get distinct colors.
\end{itemize}
\end{theorem}
We get that each $U_j$ is of cardinality $\lfloor\cd^p(\hyp)/p\rfloor$ or $\lceil\cd^p(\hyp)/p\rceil$.

Note that Theorem~\ref{thm:main} implies directly K\v{r}\'i\v{z}'s theorem when $q$ is a prime number $p$: each color may appear at most $p-1$ times within the vertices and there are $\cd^p(\hyp)$ vertices. There is a standard derivation of K\v{r}\'i\v{z}'s theorem for any $q$ from the prime case, see \cite{Zi02,Zi06}. Theorem~\ref{thm:main} is a generalization of Simonyi-Tardos's theorem except for a slight loss: when $p=2$, we do not recover the alternation of the colors between the two parts. 

Whether Theorem~\ref{thm:main} is true for non-prime $p$ is an open question.

\section{Local chromatic number and Kneser hypergraphs}\label{sec:loc}

In a graph $G=(V,E)$, the {\em closed neighborhood} of a vertex $u$, denoted $N[u]$, is the set $\{u\}\cup\{v:\,uv\in E\}$.
The {\em local chromatic number} of a graph $G=(V,E)$, denoted $\chi_{\ell}(G)$, is the maximum number of colors appearing in the closed neighborhood of a vertex minimized over all proper colorings: $$\chi_{\ell}(G)=\min_{c}\max_{v\in V}|c(N[v])|,$$ where the minimum is taken over all proper colorings $c$ of $G$. This number has been defined in 1986 by Erd\H{o}s, F\"uredi,
Hajnal, Komj\'ath, R\"odl, and Seress~\cite{ErFuHaKoRo86}. For Kneser graphs, we have the following theorem, which is a consequence of the Simonyi-Tardos theorem: any vertex of the part with $\lfloor r/2\rfloor$ vertices in the completely multicolored complete bipartite subgraph has at least $\lceil r/2\rceil+1$ colors in its closed neighborhhod (where $r=\cd^2(\hyp)$).

\begin{theorem-cite}[Simonyi-Tardos theorem for local chromatic number]
Let $\hyp$ be a hypergraph and assume that $\emptyset$ is not an edge of $\hyp$. If $\cd^2(\hyp)\geq 2$, then $$\chi_{\ell}(\KG^2(\hyp))\geq\left\lceil\frac{\cd^2(\hyp)}{2}\right\rceil+1.$$
\end{theorem-cite}
 
Note that we can also see this theorem as a direct consequence of Theorem 1 in~\cite{SiTa06} (with the help of Theorem 1 in~\cite{MaZi04}).

We use the following natural definition for the local chromatic number $\chi_{\ell}(\hyp)$ of a uniform hypergraph $\hyp=(V,E)$. For a subset $X$ of $V$, we denote by $\mathcal{N}(X)$ the set of vertices $v$ such that $v$ is the sole vertex outside $X$ for some edge in $E$: $$\mathcal{N}(X)=\{v:\,\exists e\in E\mbox{ s.t. }e\setminus X=\{v\}\}.$$ We define furthermore $\mathcal{N}[X]:=X\cup\mathcal{N}(X)$. Note that if the hypergraph is a graph, $\mathcal{N}[\{v\}]=N[v]$ for any vertex $v$. The definition of the local chromatic number of a hypergraph is then:
$$\chi_{\ell}(\hyp)=\min_{c}\max_{e\in E,\,v\in e}|c(\mathcal{N}[e\setminus\{v\}])|,$$ where the minimum is taken over all proper colorings $c$ of $\hyp$. When the hypergraph $\hyp$ is a graph, we get the usual notion of local chromatic number for graphs.

The following theorem is a consequence of Theorem~\ref{thm:main} and generalizes the Simonyi-Tardos theorem for local chromatic number to Kneser hypergraphs.

\begin{theorem}\label{thm:lochyp}
Let $\hyp$ be a hypergraph and assume that $\emptyset$ is not an edge of $\hyp$. Then 
$$\chi_{\ell}(\KG^p(\hyp))\geq\min\left(\left\lceil\frac{\cd^p(\hyp)}{p}\right\rceil+1,\left\lceil\frac{\cd^p(\hyp)}{p-1}\right\rceil\right)$$ for any prime number $p$.
\end{theorem}

\begin{proof}
Let $c$ be any proper coloring of $\KG^p(\hyp)$. Consider the complete $p$-uniform $p$-partite hypergraph $\mathcal{G}$ in $\KG^p(\hyp)$ whose existence is ensured by Theorem~\ref{thm:main}. Choose $U_j$ of cardinality $\lceil\cd^p(\hyp)/p\rceil$.
 
If $\lceil\cd^p(\hyp)/(p-1)\rceil>\lceil\cd^p(\hyp)/p\rceil$, then there is a vertex $v$ of $\mathcal{G}$ not in $U_j$ whose color is distinct of all colors used in $U_j$. Choose any edge $e$ of $\mathcal{G}$ containing $v$ and let $u$ be the unique vertex of $e\cap U_j$. We have then $|c(\mathcal{N}[e\setminus\{u\}])|\geq|U_j|+1=\lceil\cd^p(\hyp)/p\rceil+1.$

Otherwise, $\lceil\cd^p(\hyp)/(p-1)\rceil=\lceil\cd^p(\hyp)/p\rceil$, and for any edge $e$, we have $|c(\mathcal{N}[e\setminus\{u\}])|\geq\lceil\cd^p(\hyp)/p\rceil=\lceil\cd^p(\hyp)/(p-1)\rceil$, with $u$ being again the unique vertex of $e\cap U_j$.
\end{proof}

As for Theorem~\ref{thm:main}, we do not know whether this theorem remains true for non-prime $p$.

\section{Combinatorial topology and proof of the main result}\label{sec:main}

\subsection{Tools of combinatorial topology}

\subsubsection{Basic definitions}

We use the cyclic and muliplicative group $Z_q=\{\omega^j: j=1,\ldots,q\}$ of the $q$th roots of unity. We emphasize that $0$ is not considered as an element of $Z_q$. For a vector $X=(x_1,\ldots,x_n)\in(Z_q\cup\{0\})^n$, we define $X^{j}$ to be the set $\{i\in[n]:\,x_i=\omega^j\}$ and $|X|$ to be the quantity $|\{i\in[n]:\,x_i\neq 0\}|$.

An (abstract) simplicial complex $\mathsf{K}$ is a collection of subsets of a finite set $V(\mathsf{K})$, called the {\em vertex set}, such that whenever $\sigma\in\mathsf{K}$ and $\tau\subseteq\sigma$, we have $\tau\in\mathsf{K}$. Such a $\tau$ is called a {\em face} of $\sigma$. A simplicial complex is said to be {\em pure} if all maximal simplices for inclusion have same dimension. In the sequel, all simplicial complexes are abstract and we omit this specification from now on.

\subsubsection{Chains and chain maps}

Let $\mathsf K$ be a simplicial complex. We denote its chain complex by $\mathcal{C}(\mathsf{K})$. We always assume that the coefficients are taken in $\mathbb{Z}$.

\subsubsection{Special simplicial complexes}

For a simplicial complex $\mathsf{K}$, its first barycentric subdivision is denoted by $\sd(\mathsf{K})$. It is the simplicial complex whose vertices are the nonempty simplices of $\mathsf{K}$ and whose simplices are the collections of simplices of $\mathsf{K}$ that are pairwise comparable for $\subseteq$ (these collections are usually called {\em chains} in the poset terminology, with a different meaning as the one used above in ``chain complexes'').

As a simplicial complex, $Z_q$ is seen as being $0$-dimensional and with $q$ vertices. $Z_q^{*d}$ is the join of $d$ copies of $Z_q$. It is a pure simplicial complex of dimension $d-1$. A vertex $v$ taken is the $\mu$th copy of $Z_q$ in $Z_q^{*d}$ is also written $(\epsilon,\mu)$ where $\epsilon\in Z_q$ and $\mu\in[d]$. Sometimes, $\epsilon$ is called the {\em sign} of the vertex, and $\mu$ its {\em absolute value}. This latter quantity is denoted $|v|$.

The simplicial complex $\sd(Z_q^{*d})$ plays a special role. We have $V\left(\sd(Z_q^{*d})\right)\simeq\left(Z_q\cup\{0\}\right)^d\setminus\{(0,\ldots,0)\}$: a simplex $\sigma\in Z_q^{*d}$ corresponds to the vector $X=(x_1,\ldots,x_d)\in\left(Z_q\cup\{0\}\right)^d$ with $x_{\mu}=\epsilon$ for all $(\epsilon,\mu)\in\sigma$ and $x_{\mu}=0$ otherwise.

We denote by $\sigma_{q-2}^{q-1}$ the simplicial complex obtained from a $(q-1)$-dimensional simplex and its faces by deleting the maximal face. It is hence a $(q-2)$-dimensional pseudomanifold homeomorphic to the $(q-2)$-sphere. We also identify its vertices with $Z_q$. A vertex of the simplicial complex $\left(\sigma_{q-2}^{q-1}\right)^{*d}$ is again denoted by $(\epsilon,\mu)$ where $\epsilon\in Z_q$ and $\mu\in[d]$. For $\epsilon\in Z_q$ and a simplex $\tau$ of $\left(\sigma_{p-2}^{p-1}\right)^{*d}$, we denote by $\tau^{\epsilon}$ the set of all vertices of $\tau$ having $\epsilon$ as sign, i.e. $\tau^{\epsilon}:=\{(\omega,\mu)\in \tau:\,\omega=\epsilon\}$. Note that if $q$ is a prime number, $Z_q$ acts freely on $\sigma_{q-2}^{q-1}$.

\subsubsection{Barycentric subdivision operator}

Let $\mathsf{K}$ be a simplicial complex. There is a natural chain map $\sd_{\#}:\mathcal{C}(\mathsf{K})\rightarrow\mathcal{C}(\sd(\mathsf{K}))$ which, when evaluated on a $d$-simplex $\sigma\in\mathsf{K}$, returns the sum of all $d$-simplices in $\sd (\mathsf{K})$ contained in $\sigma$, with the induced orientation. ``Contained'' is understood according to  the geometric interpretation of the barycentric subdivision. If $\mathsf{K}$ is a free $Z_q$-simplicial complex, $\sd_{\#}$ is a $Z_q$-equivariant map.

\subsubsection{The $Z_q$-Fan lemma}
The following lemma plays a central role in the proof of Theorem~\ref{thm:main}. It is proved (implicitely and in a more general version) in \cite{HaSaScZi09,Me05}. 

\begin{lemma}[$Z_q$-Fan lemma]\label{ZpFan}
Let $q\geq 2$ be a positive integer.
Let $\lambda_{\#}: \mathcal{C}\left(\sd(Z_q^{*n})\right)\rightarrow \mathcal{C}\left(Z_q^{*m}\right)$ be a $Z_q$-equivariant chain map. Then there is an $(n-1)$-dimensional simplex $\rho$ in the support of $\lambda_{\#}(\rho')$, for some $\rho'\in\sd(Z_q^{*n})$, of the form $\{(\epsilon_1,\mu_1),(\epsilon_2,\mu_2),\ldots,(\epsilon_n,\mu_n)\}$, with $\mu_i<\mu_{i+1}$ and $\epsilon_i\neq\epsilon_{i+1}$ for $i=1,\ldots,n$.
\end{lemma}

This $\rho'$ is an {\em alternating simplex}.

\begin{proof}
The proof is exactly the proof of Theorem 5.4 (p.415) of \cite{HaSaScZi09}. The complex $X$ in the statement of this Theorem 5.4 is our complex $\sd(Z_q^{*n})$, the dimension $r$ is $n-1$, and the generalized $r$-sphere $(x_i)$ is any generalized $(n-1)$-sphere of $\sd(Z_q^{*n})$ with $x_0$ reduced to a single point.
The chain map $h_{\bullet}^{\ell}$ is induced by our chain map $\lambda_{\#}$, instead of being induced by the chain map $\ell_{\#}$ of \cite{HaSaScZi09} (itself induced by the labeling $\ell$). It does not change the proof since $h_{\bullet}^{\lambda}$ only uses the fact that $\ell_{\#}$ is a $Z_q$-equivariant chain map. In the statement of Theorem 5.4 of \cite{HaSaScZi09}, $\alpha_i$ is always a lower bound on the number of ``alternating patterns'' (i.e. simplices $\rho'$ as in the statement of the lemma) in $\ell_{\#}(x_i)$, even for odd $i$ since the map $f_i$ in Theorem 5.4 of \cite{HaSaScZi09} is zero on non-alternating elements. Since $\alpha_0=1$, we get that $\alpha_i\neq 0$ for all $0\leq i\leq n-1$.
\end{proof}

In particular, for $q=2$, it gives the Ky Fan theorem~\cite{Fa52} used for instance in \cite{Fa82,Me05bis,SiTa06} to derive properties of Kneser graphs.

\subsection{Proof of the main result}

\begin{proof}[Proof of Theorem~\ref{thm:main}] The proof goes as follows. We assume given a proper coloring $c$ of $\KG^p(\hyp)$. With the help of the coloring $c$, we build a $Z_p$-equivariant chain map $\psi_{\#}:\mathcal{C}(\sd(Z_p^{*n}))\rightarrow\mathcal{C}(Z_p^{*m})$, where $m=n-\cd^p(\hyp)+t(p-1)$. We apply Lemma~\ref{ZpFan} to get the existence of some alternating simplex $\rho'$ in $\sd(Z_p^{*n})$. Using properties of $\psi_{\#}$ (especially the fact that it is a composition of maps in which simplicial maps are involved), we show that this alternating simplex provides a complete $p$-uniform $p$-partite hypergraph in $\hyp$ with the required properties. 

Let $r=\cd^p(\hyp)$. We denote $\mathsf{L}:= Z_p^{*(n-r)}$ and $\mathsf{M}:=\left(\sigma_{p-2}^{p-1}\right)^{*t}$. 
Following the ideas of \cite{Ma04,Zi02}, we define $f:\left(Z_p\cup\{0\}\right)^n\setminus\{(0,\ldots,0)\}\rightarrow Z_p\times[m]$ with $m=n-r+t(p-1)$. \\

\noindent {\bf\em If $X\in\left(Z_p\cup\{0\}\right)^n\setminus\{(0,\ldots,0)\}$ is such that $|X|\leq n-r$.} Then $f(X):=(\epsilon,|X|)$ with $\epsilon$ is the first nonzero component in $X$. 
\medskip

\noindent {\bf\em If $X\in\left(Z_p\cup\{0\}\right)^n\setminus\{(0,\ldots,0)\}$ is such that $|X|\geq n-r+1$.} By definition of the colorability defect, at least one of the $X^j$ with $j\in[p]$, contains an edge of $\hyp$. Choose $j\in[p]$ such that there is $S\subseteq X^j$ with $S\in E(\hyp)$. Its defines $F(X):=S$ and $f(X):=(\epsilon,n-r+c(F(X)))$.\\

Note that $f$ induces a $Z_p$-equivariant simplicial map $f:\sd(Z_p^{*n})\rightarrow\mathsf{L}*\mathsf{M}$. \\

Let $W_{\ell}$ be the set of simplices $\tau\in\mathsf{M}$ such that $|\tau^{\epsilon}|=0$ or $|\tau^{\epsilon}|=\ell$ for all $\epsilon\in Z_p$. Let $W=\bigcup_{\ell=1}^mW_{\ell}$. Choose an arbitrary equivariant map $s:W\rightarrow Z_p$. Such a map can be easily built by choosing one representative in each orbit ($Z_p$ acts freely on each $W_{\ell}$). We build also an equivariant map $s_0:\sigma_{p-2}^{p-1}\rightarrow Z_p$, again by choosing one representative in each orbit of the action of $Z_p$. \\

We define now a simplicial map $g:\sd(\mathsf{L}*\mathsf{M}))\rightarrow Z_p^{*m}$ as follows.

Take a vertex in $\sd(\mathsf{L}*\mathsf{M})$. It is of the form $\sigma\cup\tau\neq\emptyset$ where $\sigma\in \mathsf{L}$ and $\tau\in\mathsf{M}$. \\

\noindent {\bf\em If $\tau\neq\emptyset$.} Let $\alpha:=\min_{\epsilon\in Z_p}|\tau^{\epsilon}|$. 
\begin{itemize}
\item If $\alpha=0$, define $\bar{\tau}:=\{\epsilon\in Z_p:\,\tau^{\epsilon}=\emptyset\}$ and $g(\sigma\cup\tau)=(s_0(\bar{\tau}),n-r+|\tau|)$ (we have indeed $\bar{\tau}\in\sigma_{p-2}^{p-1}$). 
\item If $\alpha>0$, define $\bar{\tau}:=\bigcup_{\epsilon:\,|\tau^{\epsilon}|=\alpha}\tau^{\epsilon}$ and $g(\sigma\cup\tau):=(s(\bar{\tau}),n-r+|\tau|)$.
\end{itemize}
The definition of $\bar{\tau}$ is illustrated on Figures~\ref{fig:tau} and~\ref{fig:tau_bar}.

\medskip

\noindent {\bf\em If $\tau=\emptyset$.} Choose $(\epsilon,\mu)$ in $\sigma$ with maximal $\mu$. Define $g(\sigma\cup\tau):=(\epsilon,\mu)$. Note that $\mathsf{L}$ is such that there is only one $\epsilon$ for which the maximum is attained.

\medskip

\begin{figure}
\includegraphics[width=8cm]{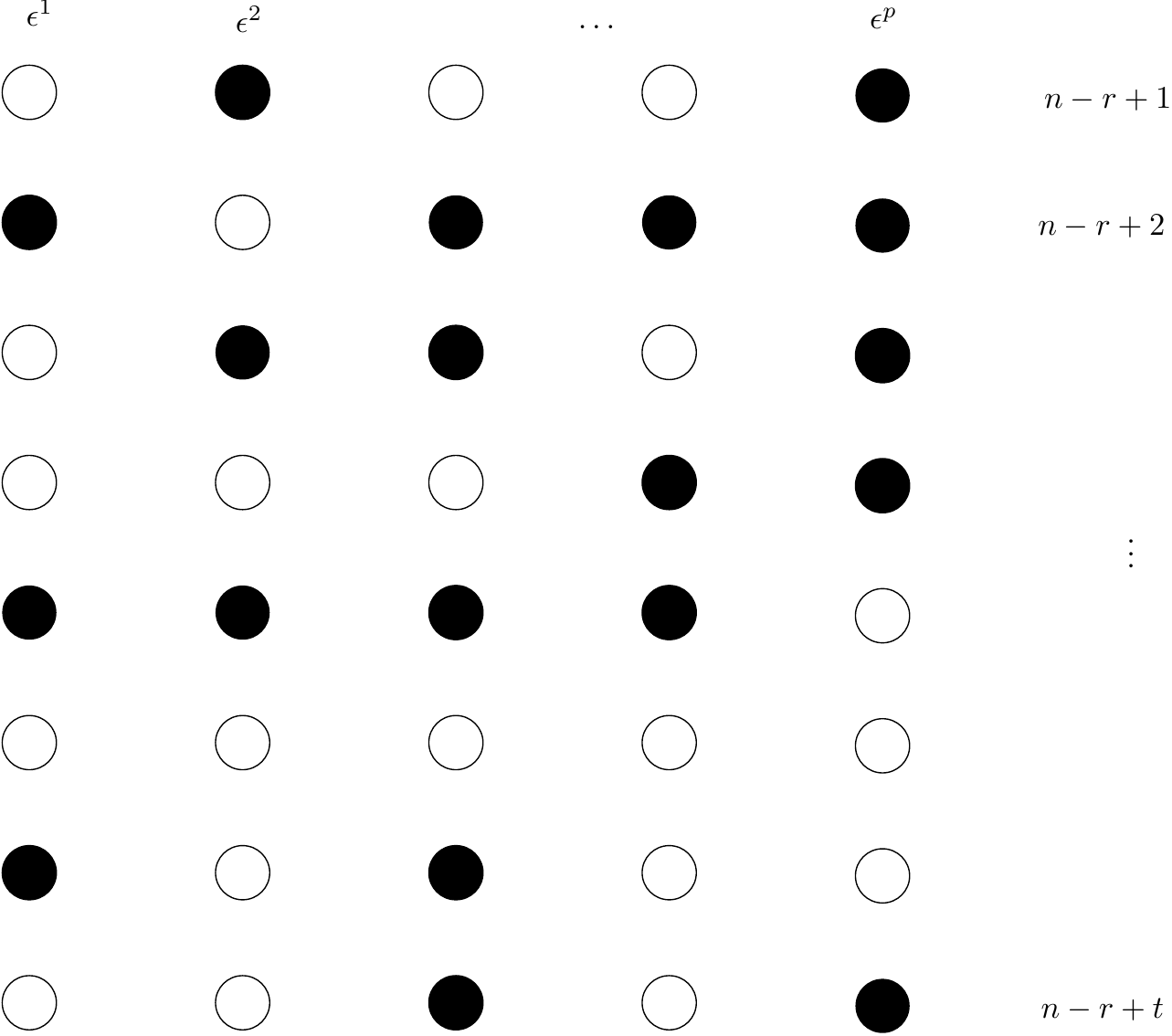}
\caption{An example of a simplex $\tau\in\mathsf{M}$.}
\label{fig:tau}
\end{figure}
\begin{figure}
\includegraphics[width=8cm]{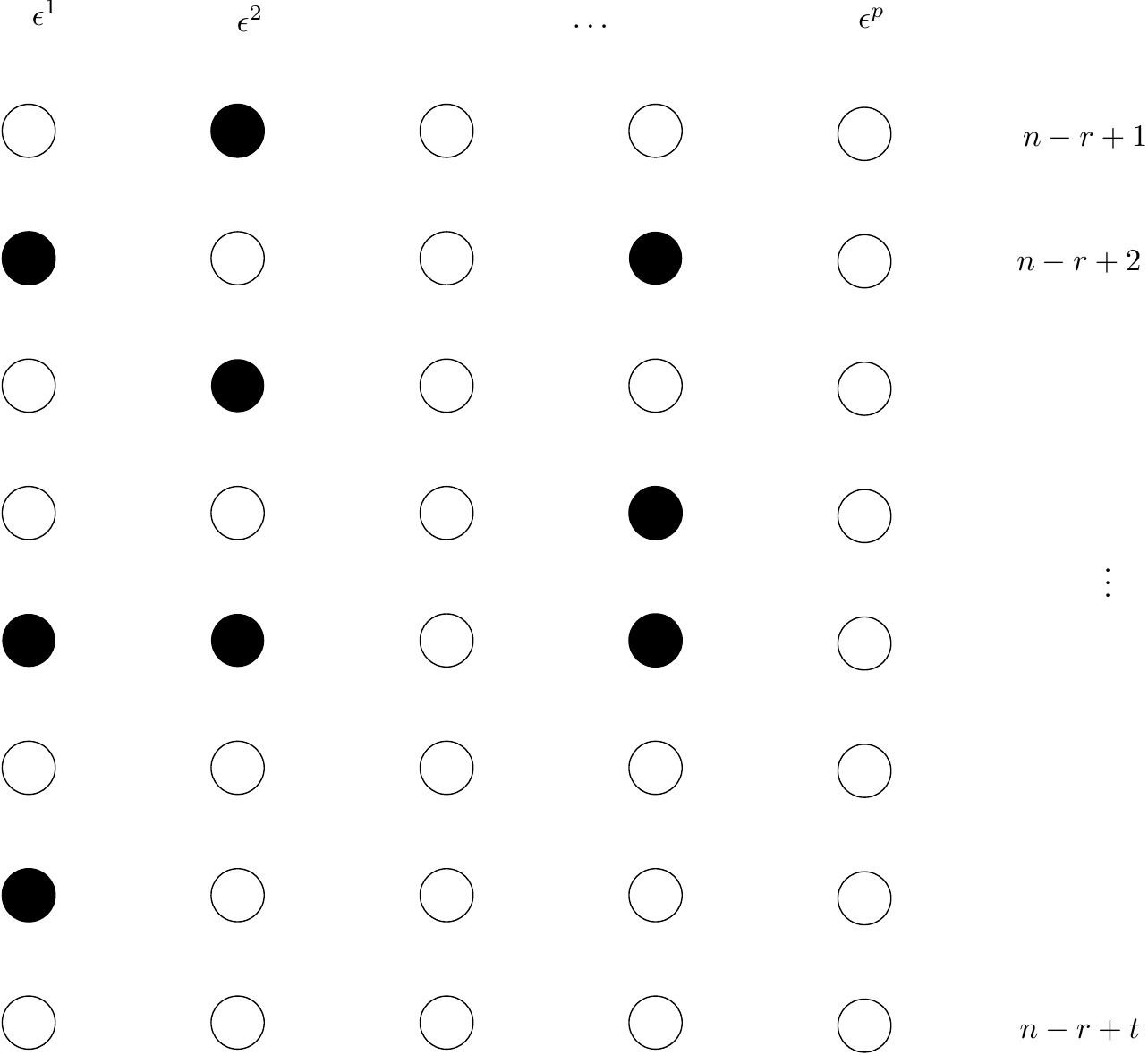}
\caption{The simplex $\bar{\tau}$ which leads to the definition of $g$.}
\label{fig:tau_bar}
\end{figure}

We check now that $g$ is a simplicial map. Assume for a contradiction that there are $\sigma\subseteq\sigma'$, $\tau\subseteq\tau'$ such that $g(\sigma\cup\tau)=(\epsilon,\mu)$ and $g(\sigma'\cup\tau')=(\epsilon',\mu)$ with $\epsilon\neq\epsilon'$. If $\tau=\emptyset$, then $\mu\leq n-r$ and $\tau'=\emptyset$. We should then have $\epsilon=\epsilon'$, which is impossible. If $\tau\neq\emptyset$, then $|\tau|=|\tau'|$, and thus $\tau=\tau'$. We should again have $\epsilon=\epsilon'$ which is impossible as well.

Moreover, $g$ is increasing: for $\sigma\subseteq\sigma'$ and $\tau\subseteq\tau'$, we have $|g(\sigma\cup\tau)|\leq|g(\sigma'\cup\tau')|$. \\

We get our map $\psi_{\#}$ by defining: $\psi_{\#}=g_{\#}\circ\sd_{\#}\circ f_{\#}$. It is a $Z_p$-equivariant chain map from 
$\mathcal{C}(\sd(Z_p^{*n}))$ to $\mathcal{C}(Z_p^{*m})$. 

This chain map $\psi_{\#}$ satisfies the condition of Lemma~\ref{ZpFan}. Hence, there exists $\rho\in Z_p^{*m}$ of the form $\rho=\{(\epsilon_1,\mu_1),\ldots,(\epsilon_n,\mu_n)\}$ with $\mu_i<\mu_{i+1}$ and $\epsilon_{i}\neq\epsilon_{i+1}$ for $i=1,\ldots,n-1$ such that $\rho$ is in the support of $\psi_{\#}(\rho')$ for some $\rho'\in\sd(Z_p^{*n})$.

Since $g$ is a simplicial map, we know that there is a permutation $\pi$ and a sequence $\sigma_{\pi(1)}\cup\tau_{\pi(1)}\subseteq\ldots\subseteq\sigma_{\pi(n)}\cup\tau_{\pi(n)}$ of simplices of $\mathsf{L}*\mathsf{M}$ such that $g(\sigma_i\cup\tau_i)=(\epsilon_i,\mu_i)$ with $\mu_i<\mu_{i+1}$ and $\epsilon_{i}\neq\epsilon_{i+1}$ for $i=1,\ldots,n-1$. Since $g$ is increasing, we get that $\pi(i)=i$ for all $i$. Using the fact that $f$ is simplicial, we get moreover that $|\sigma_{n}\cup\tau_{n}|=n$, and then that $|\sigma_{i}\cup\tau_{i}|=i$.

The fact that all $\mu_i$ are distinct implies that $\tau_i=0$ for $i=1,\ldots,n-r$. Indeed, $\tau_i=\tau_{i+1}$ implies then that $\tau_i=\emptyset$.
We have therefore $\tau_1=\cdots=\tau_{n-r}=\emptyset$ and thus $|\sigma_{n-r}|=n-r$. It implies that $|\tau_{n-r+l}|=n-r+l-|\sigma_{n-r+l}|\leq l$. On the other hand, we have that $|\tau_{n-r+l}|=\mu_{n-r+l}-n+r\geq l$. Thus, $|\tau_{n-r+l}|=l$.

Consider the sequence $(\omega_1,\nu_1),\ldots,(\omega_{n-r},\nu_{n-r})$, where $(\omega_1,\nu_1)$ is the unique vertex of $\tau_{r+1}$ and $(\omega_{l+1},\nu_{l+1})$ the unique vertex of $\tau_{r+l+1}\setminus\tau_{r+l}$ for $l=1,\ldots,n-1-r$. The sign $\omega_{l+1}$ is necessarily such that $\tau_{r+l}^{\omega_{l+1}}$ has a minimum cardinality among the $\tau_{r+l}^{\epsilon}$, otherwise the set of $\epsilon$ for which $|\tau_{r+l+1}^{\epsilon}|$ is minimum would be the same as for $|\tau_{r+l}^{\epsilon}|$, and, according to the definition of the maps $s$ and $s_0$, we would have $\epsilon_{l+1}=\epsilon_l$.

We clearly have $\left||\tau_{r+1}^{\epsilon}|-|\tau_{r+1}^{\epsilon'}|\right|\leq 1$ for all $\epsilon,\epsilon'$. Now assume that for $k\geq r+1$ we have $\left||\tau_k^{\epsilon}|-|\tau_k^{\epsilon'}|\right|\leq 1$ for all $\epsilon,\epsilon'$. Since the element added to $\tau_k$ to get $\tau_{k+1}$ is added to a $\tau_k^{\epsilon}$ with minimum cardinality, we have $\left||\tau_{k+1}^{\epsilon}|-|\tau_{k+1}^{\epsilon'}|\right|\leq 1$ for all $\epsilon,\epsilon'$. 
By induction we have in particular \begin{equation}\label{eq:balanced}\left||\tau_n^{\epsilon}|-|\tau_n^{\epsilon'}|\right|\leq 1\quad\mbox{for all } \epsilon,\epsilon'.\end{equation}

 Using the fact that $f$ is simplicial, we get a sequence $X_{n-r+1}\subseteq\ldots\subseteq X_n$ of signed vectors whose image by $f$ is $\tau_n$. Each $X_i$ provides a vertex $F(X_i)$ of $\KG^p(\hyp)$. 
For each $j$, define $U_j$ to be the set of $F(X_i)$ such that the sign of $f(X_i)$ is $\omega^j$. The $U_j$ are subsets of vertices of $\KG^p(\hyp)$ and pairwise disjoint. Moreover, for two distinct $j$ and $j'$, if $F(X_i)\in U_j$ and $F(X_{i'})\in U_j'$, we have $F(X_i)\cap F(X_{i'})=\emptyset$. Thus, the $U_j$ induces in $\KG^p(\hyp)$ a complete $p$-partite $p$-uniform hypergraph with $r=\cd^p(\hyp)$ vertices. Equation~\eqref{eq:balanced} indicates that the cardinalities of the $U_j$ differ by at most one. Since $|\tau_n|=n-r$, each $U_j$ has all its vertices of distinct colors.
\end{proof}

\section{Alternation number}\label{sec:alt}

\subsection{Definition}

Alishahi and Hajiabolhassan~\cite{AlHa13}, going on with ideas introduced in~\cite{Me11}, defined the {\em $q$-alternation number} of an hypergraph $\alt^q(\hyp)$ as an improvement of the $q$-colorability defect. It is defined as follows.

Let $q$ and $n$ be positive integers. An {\em alterning sequence} is a sequence $s_1,s_2,\ldots,s_n$ of elements of $Z_q$ such that $s_i\neq s_{i+1}$ for all $i=1,\ldots,n-1$. For a vector $X=(x_1,\ldots,x_n)\in(Z_q\cup\{0\})^n$ and a permutation $\pi\in\mathcal{S}_n$, we denote $\alt_{\pi}(X)$ the maximum length of an alternating subsequence of the sequence $x_{\pi(1)},\ldots,x_{\pi(n)}$. Note that by definition this subsequence has no zero element.

\begin{example}
Let $n=9$, $q=3$, and $X=(\omega^2,\omega^2,0,0,\omega^1,\omega^3,0,\omega^3,\omega^2)$, we have $\alt_{\id}(X)=4$. If $\pi$ is a permutation acting only on the first four positions, then $\alt_{\id}(X)=\alt_{\pi}(X)$. If $\pi$ exchanges the last two elements of $X$, we have $\alt_{\pi}(X)=5$.
\end{example}

Let $\hyp=(V,E)$ be a hypergraph with $n$ vertices. We identify $V$ and $[n]$. The {\em $q$-alternation number} of an hypergraph $\alt^q(\hyp)$ with $n$ vertices is defined as:
\begin{equation}\label{eq:alt}\alt^q(\hyp)=\min_{\pi\in\mathcal{S}_n}\max\{\alt_{\pi}(X):\,X\in(Z_q\cup\{0\})^n\mbox{ with }\,E(\hyp[X^{j}])=\emptyset\mbox{ for $j=1,\ldots,q$}\}.\end{equation} Note that this number does not depend on the way $V$ and $[n]$ have been identified.

\subsection{Improving the results with the alternation number}

Alishahi and Hajiabolhassan improved the K\v{r}\'i\v{z} theorem by the following theorem.

\begin{theorem-cite}[Alishahi-Hajiabolhassan theorem]
Let $\hyp$ be a hypergraph and assume that $\emptyset$ is not an edge of $\hyp$. Then $$\chi(\KG^q(\hyp))\geq\left\lceil\frac{|V(\hyp)|-\alt^q(\hyp)}{q-1}\right\rceil$$ for any integer $q\geq 2$.
\end{theorem-cite}

Theorem~\ref{thm:main} and Theorem~\ref{thm:lochyp} can be similarly improved with the alternation number. Let $\pi$ be the permutation on which the minimum is attained in Equation~\eqref{eq:alt}. We replace $r=\cd^p(\hyp)$ by $r=|V(\hyp)|-\alt^p(\hyp)$ in the both proofs and $|X|$ in the definition of $f$ by $\alt_{\pi}(X)$ in the proof of Theorem~\ref{thm:main} without any other change. Since we have $|V(\hyp)|-\alt^p(\hyp)\geq\cd^p(\hyp)$, often with a strict inequality -- see \cite{AlHa13} -- it improves these theorems.

\begin{theorem}\label{thm:main-bis}
Let $\hyp$ be a hypergraph and assume that $\emptyset$ is not an edge of $\hyp$. Let $p$ be a prime number. Then any proper coloring $c$ of $\KG^p(\hyp)$ with colors $1,\ldots,t$ ($t$ arbitrary) must contain a complete $p$-uniform $p$-partite hypergraph with parts $U_1,\ldots,U_p$ satisfying the following properties. 
\begin{itemize}
\item It has $|V(\hyp)|-\alt^p(\hyp)$ vertices. 
\item The values of $|U_j|$ for $j=1,\ldots,p$ differ by at most one.
\item For any $j$, the vertices of $U_j$ get distinct colors.
\end{itemize}
\end{theorem}

\begin{theorem}\label{thm:lochyp-bis}
Let $\hyp$ be a hypergraph and assume that $\emptyset$ is not an edge of $\hyp$. Then 
$$\chi_{\ell}(\KG^p(\hyp))\geq\min\left(\left\lceil\frac{|V(\hyp)|-\alt^p(\hyp)}{p}\right\rceil+1,\left\lceil\frac{|V(\hyp)|-\alt^p(\hyp)}{p-1}\right\rceil\right)$$ for any prime number $p$.
\end{theorem}

\subsection{Complexity}

It remains unclear whether the alternation number, or a good upper bound of it, can be computed efficiently. However, we can note that given a hypergraph $\hyp$, computing the alternation number for a fixed permutation is an NP-hard problem. 

\begin{proposition} 
Given a hypergraph $\hyp$, a permutation $\pi$, and a number $q$, computing 
$$\max\{\alt_{\pi}(X):\,X\in(Z_q\cup\{0\})^n\mbox{ with }\,E(\hyp[X^{j}])=\emptyset\mbox{ for $j=1,\ldots,q$}\}$$ 
is NP-hard.
\end{proposition}

\begin{proof}
The proof consists in proving that the problem of finding a maximum independent set in a graph can be polynomially reduced to our problem for $q=2$, $\pi=\id$, and $\hyp$ being some special graph.

Let $G$ be a graph. Define $G'$ to be a copy of $G$ and consider the {\em join} $\hyp$ of $G$ and $G'$. The join of two graphs is the disjoint union of the two graphs plus all edges $vv'$ with $v$ a vertex of $G$ and $v'$ a vertex of $G'$. We number the vertices of $G$ arbitrarily with a bijection $\rho:V\rightarrow[|V|]$.
It gives the following numbering for the vertices of $\hyp$. In $\hyp$, a vertex $v$ receives number $2\rho(v)-1$ and its copy $v'$ receives the number $2\rho(v)$. Let $n=2|V|$. As usual, we denote the maximum cardinality of an independent set of $G$ by $\alpha(G)$. 

Let $I\subseteq V$ be a independent set of $G$. Define $X\in(Z_2\cup\{0\})^n$ as follows: 
$$X_{2\rho(v)-1}=+1\mbox{ and }X_{2\rho(v)}=-1\mbox{ for all $v\in I$, and } x_i=0\mbox{ for the other indices $i$.}$$ By definition of the numbering, we have $\alt_{\id}(X)=2|I|$ and thus $$\max\{\alt_{\id}(X):\,X\in(Z_2\cup\{0\})^n\mbox{ with }\,E(\hyp[X^{j}])=\emptyset\mbox{ for $j=1,2$}\}\geq 2\alpha(G)$$

Conversely, any $X\in(Z_2\cup\{0\})^n$ gives an independent set $I$ in $G$ and another $I'$ in $G'$: take a longest alternating subsequence in $X$ and define the set $I$ as the set of vertices $v$ such that $X_{2\rho(v)-1}\neq 0$ and the set $I'$ as the set of vertices $v$ such that $X_{2\rho(v)}\neq 0$. 
 We have $\alt_{\id}(X)=|I|+|I'|$ because two components of $X$ with distinct index parities cannot be of opposite signs: each vertex of $G$ is the neighbor of each vertex of $G'$. Thus 
$$\max\{\alt_{\id}(X):\,X\in(Z_2\cup\{0\})^n\mbox{ with }\,E(\hyp[X^{j}])=\emptyset\mbox{ for $j=1,2$}\}\leq 2\alpha(G).$$
\end{proof}

The same proof gives also that computing the two-colorability defect $\cd^2(\hyp)$ of any hypergraph $\hyp$ is an NP-hard problem.

\bibliographystyle{amsplain}
\bibliography{Kneser}

\end{document}